\documentclass[conference]{IEEEtran}
\IEEEoverridecommandlockouts
\usepackage{cite}
\usepackage{amsmath,amssymb,amsfonts}
\usepackage{graphicx}
\usepackage{textcomp}
\usepackage{xcolor}
\def\BibTeX{{\rm B\kern-.05em{\sc i\kern-.025em b}\kern-.08em
    T\kern-.1667em\lower.7ex\hbox{E}\kern-.125emX}}
\usepackage{tikz}
\usepackage{xcolor}
\usetikzlibrary{backgrounds}
\usepackage{setspace}
\usetikzlibrary{positioning}
\usetikzlibrary{shapes.geometric}
\usepackage{ifpdf}
\IEEEoverridecommandlockouts  
\usepackage[utf8]{inputenc}
\usepackage{amsmath}
\usepackage{graphicx}
\usepackage{flushend}
\usepackage{caption}
\usepackage{subcaption}
\usepackage{comment}
\usepackage{bm}
\usepackage{xspace}
\usepackage{soul}
\usepackage{amsopn}
\usepackage{algpseudocode}
\usepackage{algorithm} 
\usepackage{enumerate}
\makeatletter
\let\NAT@parse\undefined
\makeatother
\usepackage{amsmath,amsfonts, amssymb,amsthm,color}
\usepackage[numbers,sort&compress]{natbib}
\usepackage{amsf onts}
\usepackage{color}
\usepackage{setspace}

\usepackage{amssymb}

\newtheorem{assumption}{Assumption}

\newtheorem{lemma}{Lemma}

\usepackage{breqn}
\usepackage{hyperref} 
\usepackage{epstopdf}
\usepackage{lipsum}
\usepackage{epsfig}
\DeclareGraphicsExtensions{.pdf,.eps,.png,.jpg,.mps}

\begin{document}

\title{On Distributed Optimization in the Presence of Malicious Agents
\thanks{DISTRIBUTION A. Approved for public release. Distribution unlimited. Case Number AFRL-2021-0298. Dated 05 Feb 2021.}
}

\author{\IEEEauthorblockN{Iyanuoluwa Emiola}
\IEEEauthorblockA{\textit{ Department of Electrical Engineering} \\
\textit{University of Central Florida}\\
Orlando FL 32816, USA \\
iemiola@knights.ucf.edu}
\and
\IEEEauthorblockN{Laurent Njilla}
\IEEEauthorblockA{\textit{U.S. Air Force Research Laboratory.} \\
Rome NY USA \\
laurent.njilla@us.af.mil}
\and
\IEEEauthorblockN{Chinwendu Enyioha}
\IEEEauthorblockA{\textit{ Department of Electrical Engineering} \\
\textit{University of Central Florida}\\
Orlando FL 32816, USA \\
cenyioha@ucf.edu}
}

\maketitle

\begin{abstract}
In this paper, we consider an unconstrained distributed optimization problem over a network of agents, in which some agents are adversarial. We solve the problem via gradient-based distributed optimization algorithm and  characterize the effect of the adversarial agents on the convergence of the algorithm to the optimal solution.  
The attack model considered is such that agents locally perturb their iterates before broadcasting it to neighbors; and we analyze the case in which the adversarial agents cooperate in perturbing their estimates and the case where each adversarial agent acts independently. 
Based on the attack model adopted in the paper, we show that the solution converges to the neighborhood of the optimal solution and depends on the magnitude of the attack (perturbation) term. The analyses presented establishes conditions under which the malicious agents have enough information to obstruct convergence to the optimal solution by the non-adversarial agents. 
 \end{abstract}

\section{Introduction}
To solve an optimization problem over a network of agents in a distributed manner, gradient-based methods alongside an agreement update step are commonly used \cite{boyd2011distributed,nesterov2013introductory}. In the process, each agent iteratively updates their estimates and exchanges it with neighboring nodes. This well-studied process arrives at the optimal solution depending on certain assumptions made on the cost function being optimized and the choice of the step-size. The presence of adversarial nodes in the network who may have a different objective causes a problem and can hinder the non-adversarial nodes from arriving at an optimal solution. 
Agents in the network may act in a malicious or adversarial way either because they are faulty, or have been compromised and are being used as stooges for an undesired goal in the network. 

In typical formulations of distributed optimization problems, the objective can be considered decomposable as
\begin{equation}\label{eqn:original}
\min_{x\in \mathbb{R}^{p}} f(x)=\sum_{i=1}^n f_i(x_i),
\end{equation}
where $n$ is the total number of agents, $f_i(\cdot)$ is the local objective function of agent $i$, $x_i$ is the decision variable of each agent and $f(x)$ is the global objective function that is meant to be optimized. Each agent $i$ will also optimize its local objective function $f_i(x_i)$ and iteratively exchanges its decision variable $x_i$ with neighboring agents over a communication network. In the presence of malicious agents, however, the local and consensus computations are altered. In \cite{maybury2006detecting}, for example,  where an attempt is made to detect sensitive malicious insider threats.%


Different approaches have been taken to solve distributed optimization problems when adversarial nodes are present. An example is the approach taken in \cite{ravi2019case} where the author uses the Fast Row-stochastic Optimization with uncoordinated Step-sizes (FROST) algorithm that does not require the nodes to compute step sizes. The authors in \cite{ravi2019case} also considers the bounds on a parameter and a gradient bounds to show the strength of the attack, though an explicit characterization of the extent to which the perturbed parameter alters and prevents convergence is not presented. Another method is the topological approach in tolerating malicious nodes shown in \cite{sundaram2010distributed} where the author examines the conditions under which a malicious agent can be identified based on the topology and the size of the network. Some approaches to decision problems in the presence of adversaries often assume certain so-called `trusted' agents cannot be compromised, and use information being shared by those agents as a benchmark to identify and exclude malicious information; thus, building in resilience to their optimization algorithm \cite{baras2019trust,zhao2017resilient}.
Similarly flavored problems have been studied in the context of state estimation where methods to identify and extract malicious information are proposed \cite{sundaram2016secure, lu2020distributed}. Other related adversarial problems like \cite{marano2008distributed,yan2012vulnerability,zhang2014distributed,kailkhura2015consensus} explore the detection of attacks on distributed systems and protection strategies. 

\subsection{Contribution} 
This paper presents an analysis of the effects of malicious agents on the solution of a distributed optimization problem over a network using the gradient descent algorithm. We show how adversarial nodes can  disrupt convergence to optimal solution of the network with knowledge of the average initial value of the non-malicious agents. When the communication structure amongst the agents is a complete graph, we show how cooperation enables the agents to prevent convergence to the optimal solution by perturbing their local estimates. And when the communication structure is not a complete graph, we characterize how the malicious agents can cause disruption if they have an initial value of the regular agents estimates. We show that for the agents solving the distributed optimization problem to converge to a neighborhood of the optimal solution, the distance between their average initial value and the optimal solution has to be less than the magnitude of the attack vector. 
%

The rest of the paper follows the following structure: In Section \ref{sec:problemformulation}, the optimization problem and attack model is presented. Section \ref{sec:Convanalysiscentralized-BB} summarizes the convergence analyses and main results of the paper.  Numerical experiments follow in Section \ref{sec:Numerical} to illustrate the theoretical results. The paper ends with concluding remarks in Section \ref{sec:conclude}.

\subsection{Notation}
We respectively denote the set of positive and negative reals as $\mathbb{R}_+$ and $\mathbb{R}_-$. We denote a vector or matrix transpose as $(\cdot)^T$, and the L$2$-norm of a vector by $||\cdot||$. We also denote the gradient of a function $f(\cdot)$ as $\nabla f(\cdot)$ and an $n$ dimensional vector of ones as $1_n$.

\section{Problem formulation and Attack Model}\label{sec:problemformulation}
We consider a network comprising $n$ agents represented by an undirected graph $G = (\mathcal{V},\mathcal{E})$ where $\mathcal{V} = {1,2, ... n}$ is the set of nodes (agents) and   $\mathcal{E}={(i,j)}$ is the set of edges. Let the neighbors of each agent $i$ be denoted by the set $N_i = \{j: \ (i,j) \in \mathcal{E}\}$. Because the graph is undirected, $(i,j) \in \mathcal{E}$ also implies $(i,j) \in \mathcal{E}$. 
The agents collectively solve the unconstrained distributed optimization problem
\begin{equation}\label{eqn:quadratic}
\min_{x\in \mathcal{X}} f(x) = \sum_{i=1}^n f_i(x),
\end{equation}
where each local objective function $f_i(\cdot)$ is convex and smooth and $\mathcal{X}$ is the feasible set. 
To solve the optimization problem using the gradient descent method, each agent $i$ maintains a local copy $x_i\in \mathbb{R}^p$ of the decision variable $x\in \mathbb{R}^p$ and carries out a local update using their local cost function and broadcast the same to their neighbors:
\begin{equation}\label{eqn:gradient}
    x_i(k+1) = x_i(k) - \alpha_i \nabla f_i(x(k)),
\end{equation}
where $\alpha_{i} \in \mathbb{R}_+$ is an appropriately chosen step size. It is known that if $f_i$  is convex and differentiable, with an appropriately chosen step size $\alpha_{i}$ the updates in Equation \ref{eqn:gradient} will converge to the optimal solution \cite{nesterov2013introductory}. The problem set-up considers two cases -- the complete graph and the non-complete graph case. In the complete graph case, the malicious agents are assumed to know each other and coordinate to choice of an attack vector. We assume the adversarial agents are not known to the rest of the network \textit{a priori}. The objective of the adversarial or malicious nodes is to distort the network from reaching the true optimal solution $x^*$ of Problem \ref{eqn:quadratic}. 
To accomplish the malicious objective, rather than follow the update in Equation \ref{eqn:gradient}, the adversarial nodes perturb their local estimates with an attack vector $\epsilon \in \mathbb{R}^p$:
\begin{equation}\label{eqn:gradientmalicious}
x_i(k+1) = x_i(k) - \alpha_{i} \nabla f_i(x(k)) + \epsilon(k),
\end{equation}
before broadcasting their estimates to neighboring agents in the network. We note that an alternative formulation is to assume a different objective function 
\begin{equation} \label{eqn:maliciousminimization}
    \min_x \hat{f}(x)
\end{equation}
for the adversarial agents such that the optimal solution to $\hat{f}(x)$ is  $x^a = x^* + \hat{\epsilon}$, where $x^*$ is the optimal solution to the objective function $f(x)$. 
We assume the adversarial agents carefully pick values of $\epsilon$ by which to perturb their local estimates so that they remain undetected; and the non-adversarial nodes do not know which of their neighbors are malicious.

Next, we analyze convergence of the distributed gradient-based method to solve Problem \eqref{eqn:quadratic} using the update in \eqref{eqn:gradient} when there are malicious agents. 
Before proceeding, however, we note the assumptions being made on Problem \eqref{eqn:quadratic}.

\begin{assumption}\label{assume-bounded-G}
The decision set of agents in the network $\mathcal{X}$ is bounded. This means there exists some positive constant $0 \leq B < \infty$ such that $|\mathcal{X}|\leq B$.
\end{assumption}
\begin{assumption}\label{Assumption1}
The cost function $f(x)$ in Problems \eqref{eqn:quadratic} is strongly convex and twice differentiable.
This implies that for any vectors $x, y \in \mathbb{R}^{p}$, there exists $\mu\in\mathbb{R}_+$ such that:
\begin{equation*}
    f(x)\geq f(y)+\nabla f(y)^{T}(x-y)+\frac{\mu}{2}\|x-y\|^2.
\end{equation*}
\end{assumption}

\begin{assumption}\label{Assumption4}
The gradient of the objective function $\nabla f$ is Lipschitz continuous. This implies that for all vectors $x,y \in\mathbb{R}^n$, there exists a constant $L\in\mathbb{R}_+$ such that:
\begin{equation*}
   \|\nabla f(x)-\nabla f(y)\|\leq L\|x-y\|.
\end{equation*}
\end{assumption}
These assumptions are standard in the distributed optimization literature, as they allow for analysis. 

\section{Convergence Analysis}\label{sec:Convanalysiscentralized-BB}
We will characterize convergence for the problem and attack model presented in Section \ref{sec:problemformulation}, based on the distributed gradient descent algorithm and agreement updates. Each agent $i$ updates their local estimate $x_i$ and takes a weighted average of neighboring nodes following 
\begin{equation}\label{concatenation11}
x_i\left(k+1\right) = \sum_{j\in N_i\cup\{i\}} W_{ij}\left(x_j\left(k\right)-\alpha_{i}\nabla f_i\left(x_i\left(k\right)\right)\right),
\end{equation}
where $W$ is an $n$-dimensional square weighting matrix comprising entries $W_{ij}$ that denote the weight attached to agent $j$'s estimate by agent $i$. 

Let  $X=[x_1; \ x_2; \ \hdots \ x_n]^T \in \mathbb{R}^{np}$
be the concatenation of the local variables $x_i$, $I_p$ be the identity matrix whose dimension is $p$,  $\otimes$ be the Kronecker operation, $1_n$ be an $n$ dimensional vector of ones and let $W$ be doubly stochastic.
We can express equation \eqref{concatenation11} more compactly as:
\begin{equation}\label{concatenation}
    X(k+1)=(W \otimes I_p)X(k)-\alpha_i \nabla f(X(k)),
\end{equation}
where $W\otimes I_p \in \mathbb{R}^{np\times np}$, and $\nabla f(X(k))\in \mathbb{R}^{np}$ is the gradient of the objective $f(\cdot)$ evaluated at $X(k)$. The doubly stochastic matrix $W$ has one eigenvalue $\lambda=1$ and the other eigenvalues satisfy  $0<\lambda<1$. 

\subsection{Convergence Analysis over a Complete Graph}\label{sec:convanalysisfirststepsize}
To characterize convergence in the complete graph case, we introduce some additional notation to be used in the analyses. Let the average of local estimates be $\overline{x}(k)$ and the average of local gradients at current estimates be $\overline{g}(k)$; that is,
\[
\overline{x}(k) = \frac{1}{n}\sum_{i=1}^n x_i(k), \quad \text{and} \quad \overline{g}(k) = \frac{1}{n} \sum_{i=1}^n \nabla f_i(x_i(k)).
\]
where $\overline{x}(k)\in \mathbb{R}^{p}$ and $\overline{g}(k)\in \mathbb{R}^{p}$. 
  Based on the definition of $\overline{X}(K)$, let $\overline{X}(k)$ be such that $\overline{X}(k)=[\overline{x}(k),.....\overline{x}(k)]\in \mathbb{R}^{np}$, then according to Lemma IV.2 in  \cite{berahas2018balancing}, we have the following:
\begin{equation*}
\overline{X}(k)=\frac{1}{n}((1_n1_n^{T})\otimes I)X(k).
\end{equation*}
Since $W$ is doubly stochastic, we have:
\begin{equation}\label{averagedconcatenation}
  \overline{X}(k+1)=\frac{1}{n}((1_n1_n^{T})\otimes I)X(k).
\end{equation}
%
From Equation \eqref{averagedconcatenation} as proved in \cite{berahas2018balancing}, the consensus update can be expressed as
\[
\overline{x}(k+1)=\overline{x}(k)-\alpha \overline{g}(k).
\]
For the complete graph case, since the malicious agents are aware of one another and can cooperate, collectively deciding on the degree $\epsilon$ to which they want to perturb their local estimates for their adversarial goal. Therefore, all malicious agents choose the same $\epsilon \in \mathbb{R}^p$.
In our first result, we derive the condition under which convergence to a neighborhood of the optimal solution may be attained. As we will see, the size of the neighborhood, amongst others depends on the magnitude of the attack vector $\epsilon$.

\begin{lemma}\label{lemmacentralized3}
Suppose Assumptions \ref{assume-bounded-G}, \ref{Assumption1} and \ref{Assumption4} hold and let the perturbation parameter  $\epsilon>0$ be given. If the average initial value of of the agents when malicious agents are present satisfy $\|\overline{x}(0)-x^*\| <\|\epsilon\|$ and the step size $\alpha$ satisfies
 \[
 \alpha < \frac{2}{\mu + L} < \frac{1}{\mu} \quad \text{and} \quad \frac{\mu + L}{4\mu L} < \alpha < \frac{\mu + L}{2\mu L},
 \]
then the iterates generated converge to a neighborhood of the optimal solution, $x^*$; where $\mu$ and $L$ are the strong convexity parameter and Lipschitz constant of the objective function and its gradient respectively with $\mu \leq L$. 
\end{lemma}
\begin{proof}
The iterative equation solution for a distributed gradient descent is:
\[X(k) = -\alpha\sum_{s=0}^{k-1} \left(W^{(k-1-s)} \otimes I\right)\nabla f\left(x(s)\right),\]
which accounts for the consensus step as well. The relationship between the optimal solution $x^*$ and the desired malicious solution $x^a$ of the adversarial agents can be expressed as: $x^a=x^*+\epsilon$. 
When malicious agents are present, we have the following update:
\[\overline{x}(k+1)  - x^a = \overline{x}(k)  - \alpha \overline{g}(k) - x^a.\]
We now have the iterate equation as:
\[\overline{x}(k+1) - \left(x^* +\epsilon\right) = \overline{x}(k) - \left(x^* + \epsilon\right) - \alpha \overline{g}(k).\]
To analyze convergence of the iterates to the optimal solution, we will begin by considering the iterative equation. We can express
$\|\overline{x}(k+1) - x^* - \epsilon\|^2$ as:
\begin{equation*}
   \begin{aligned}
   \| \overline{x}(k+1) - x^* - \epsilon\|^2 ={} &  \| \overline{x}(k) - x^* - \epsilon - \alpha \overline{g}(k)\|^2, \\
   ={} &\| \overline{x}(k) - x^*\|^2 + \|\epsilon\|^2 + \alpha^2\|\overline{g}(k)\|^2 \\
     & +2\epsilon(\alpha \overline{g}(k)-(\overline{x}(k) - x^*) ) \\
    &-2 (\overline{x}(k) - x^* )^{T}(\alpha\overline{g}(k)).
   \end{aligned}
\end{equation*}
 By using vector norm principle, we know that for vectors $a$, $b$, the inequality $2 a^{T}b \leq \| a\|^2 + \| b\|^2$ is satisfied.
By similarly applying vector norm principles, we have the following:
\begin{equation*}
    \begin{aligned}
    2\epsilon (\alpha \overline{g}(k) - (\overline{x}(k)-x^*)) \leq{} &
    \|\epsilon\|^2 +\|\alpha \overline{g}(k) - (\overline{x}(k)-x^*)\|^2, \\
    ={} & \|\epsilon\|^2+\|\alpha  \overline{g}(k)\|^2+\|\overline{x}(k)-x^*\|^2  \\ 
    &-2 (\overline{x}(k) - x^* )^{T}\\
    \leq{} & \|\epsilon\|^2+\alpha^2\|\overline{g}(k)\|^2+\|\overline{x}(k)-x^*\|^2\\
    &- \alpha c_1\|\overline{g}(k)\|^2-\alpha c_2\|\overline{x}(k)-x^*\|^2,
    \end{aligned}
\end{equation*}
where the values of $c_1$ and $c_2$ are \cite{nesterov2013introductory}:
$$c_1= \frac{2}{\mu + L}\quad \text{and} \quad  c_2=\frac{2\mu L}{\mu +L}. $$ 
By using strong convexity of the objective function we obtain (Theorem 2.1.12 in \cite{nesterov2013introductory}):

\begin{equation}
\begin{aligned}
\| \overline{x}(k+1) - x^* - \epsilon \|^2 ={} & \| \overline{x}(k) - x^* - \epsilon - \alpha \overline{g}(k)\|^2, \\
\leq{} &\| \overline{x}(k) - x^*\|^2 + \| \epsilon\|^2 + \alpha^2\| \overline{g}(k)\|^2 \\ 
      & +\|\epsilon\|^2 + \alpha^2\| \overline{g}(k)\|^2 {+} \| \overline{x}(k) {-} x^*\|^2 \\
      & - \alpha c_1\|\overline{g}(k)\|^2 -\alpha c_2\|\overline{x}(k){-}x^*\|^2 \\
      & - \alpha c_1\|\overline{g}(k)\|^2-\alpha c_2\|\overline{x}(k)-x^*\|^2,\\
      ={} & (2-2\alpha c_2)\| \overline{x}(k) - x^*\|^2 \\ 
&+ (2\alpha^2-2\alpha c_1)\|\overline{g}(k)\|^2 + 2\|\epsilon\|^2.
\label{eq:bound-analyzed}
\end{aligned}
\end{equation}
In what follows we would show that the terms in the right hand side of Equation \eqref{eq:bound-analyzed} does not grow unbounded and is, in fact, related to the initial iterates and magnitude of the malicious attack. 
Clearly $\|\epsilon\|^2$ is positive and $(2\alpha^2-2\alpha c_1)$ is negative when $\alpha<c_1$.
Now we will show that $(2-2\alpha c_2)>0$ by equivalently showing that $\alpha c_2<1$.
\newline
By using the value $c_2=\frac{2\mu L}{\mu +L}$, we obtain:
\begin{equation*}
    \alpha c_2=\frac{2\alpha\mu L}{\mu +L}.
\end{equation*}
Since $\alpha <\frac{1}{\mu}$, then we obtain: 
\begin{equation*}
   \alpha c_2<\frac{1}{\mu}\frac{2\mu L}{\mu +L}=\frac{2L}{\mu +L}.
\end{equation*}
We know that both $L$ and $\mu$ are positive and $\mu \leq L$. Therefore if $\mu =L$, then, $2L/(\mu +L)=1$. So we obtain the fact that $ \alpha c_2<1.$
We have now affirmed that $(2-2\alpha c_2)>0$. Moreover, if $\alpha c_2 > \frac{1}{2}$, then we obtain that $1-\alpha c_2 < \frac{1}{2}$ and we obtain that $2-2\alpha c_2 < 1$. Therefore by using the condition:
\begin{equation*}
   0< 2-2\alpha c_2 < 1
\end{equation*}
the left hand side of Equation \eqref{eq:bound-analyzed} can be upper bounded by
\begin{equation}\label{eqn:upperbound11}
   \| \overline{x}(k+1)-x^*-\epsilon \|^2\leq(2-2\alpha c_2)\| \overline{x}(k)-x^*\|^2+2\|\epsilon\|^2.
\end{equation}
If $\overline{x}(k+1) - x^* < 0 <\epsilon$, then we have the result:
\begin{equation}\label{eqn:upperbound12}
 \| \overline{x}(k+1) - x^* - \epsilon\|^2 > \| \overline{x}(k+1) - x^* \|^2.
\end{equation}
From equations \eqref{eqn:upperbound11} and \eqref{eqn:upperbound12}, we obtain the following relationship: 
\begin{equation*}
\|\overline{x}(k+1)-x^*\|^2\leq (2-2\alpha c_2)\| \overline{x}(k)-x^*\|^2+2\|\epsilon\|^2.
\end{equation*}
By recursion we obtain:
\begin{equation*}
\| \overline{x}(k)-x^*\|^2
\leq (2-2\alpha c_2)^k\| \overline{x}(0)-x^*\|^2+2\|\epsilon\|^2,
\end{equation*}
from which we conclude
\begin{equation}\label{leftandrightbounds1}
\| \overline{x}(k)-x^*\|\leq (2-2\alpha c_2)^{\frac{k}{2}}\| \overline{x}(0)-x^*\|+\sqrt{2}\|\epsilon\|.
\end{equation}
Therefore, the iterates converge to the neighborhood of the optimal solution, $x^*$.
\end{proof}
The central idea in Lemma \ref{lemmacentralized3} is that the average initial value of the agents need to lie within $\epsilon$ of the optimal solution $x^*$ for the agents to converge to a neighborhood of the optimal solution in the presence of malicious agents. 
Knowledge of the average initial starting value is also critical for the adversarial nodes, because their choice of $\epsilon$ could depend on the initial average value of $\overline{x}$. 
The compete graph case in Lemma \ref{lemmacentralized3} also allows for the malicious agents who know one another to cooperate in choosing the attack vector or perturbation parameter $\epsilon$. Next, we consider a general case where cooperation is not as easy because of the subset of malicious agents may not be neighbors.

\subsection{Convergence Analysis over General Graph Structures}
  We consider the case in which the communication structure is more general, as opposed to being a complete graph. With a general structure, malicious agents do not necessarily have the liberty to cooperate and agree on values for the attack vector $\epsilon$, since they may not be adjacent to one another in the network. In other words, each regular agent independently solves the minimization problem \eqref{eqn:quadratic} with the malicious agents additively 
  perturbing their local estimates by the attack vector $\epsilon_i$. 
 We will now examine the conditions on the attack parameter that enables convergence when non-adversarial and malicious agents are present in a general graph structure.

We will now show conditions on $\epsilon$ that enable neighborhood convergence of iterates to the optimal point.
\begin{lemma}\label{lemmaindividualagent2}
Suppose Assumptions \ref{assume-bounded-G}, \ref{Assumption1} and \ref{Assumption4} hold, and let $\epsilon \succeq 0$. If $\|x_i(0)-x^*\| <0 <\epsilon_i \ \forall i$ and the step size $\alpha$ satisfies
 \[
 \alpha < \frac{2}{\mu + L} < \frac{1}{\mu} \quad \text{and} \quad \frac{\mu + L}{4\mu L} < \alpha < \frac{\mu + L}{2\mu L},
 \]
then the individual iterates generated converge to the neighborhood of the optimal solution, $x^*$, where $\mu$ and $L$ are the strong convexity parameter and Lipschitz constant of the objective function and its gradient respectively with $\mu \leq L$.
\end{lemma}
\begin{proof}
The proof is similar to the one in Lemma \ref{lemmacentralized3} except that in this scenario, each agent is individually solving its own problem. In this case, the malicious agents are not cooperating to coordinate the attack vector $\epsilon$. We begin with the iterate equation:
\begin{equation}
   \begin{aligned} \label{eq:individual-distance}
   \| x_i(k+1) - x^* - \epsilon_{i}\|^2 ={} &  \| x_i(k) - x^* - \epsilon_{i} - \alpha g(k)\|^2, \\
   ={} &\| x_i(k) - x^*\|^2 + \|\epsilon_{i}\|^2 + \alpha^2\|g(k)\|^2 \\
     & +2\epsilon_{i}(\alpha g(k)-(x_i(k) - x^*) ) \\
    &-2 (x_i(k) - x^* )^{T}(\alpha g(k)).
   \end{aligned}
\end{equation}
Leveraging the fact that for vectors $ a, b$, the inequality 
\[2 a^{T}b \leq \| a\|^2 + \| b\|^2\]
holds, we can further simplify the fourth summand in Equation \eqref{eq:individual-distance} as
\begin{equation*}
    \begin{aligned}
    2\epsilon_{i} (\alpha g(k) - (x_i(k)-x^*)) \leq{} &
    \|\epsilon_{i}\|^2 +\|\alpha g(k) - (x_i(k)-x^*)\|^2, \\
    ={} & \|\epsilon_{i}\|^2+\|\alpha  g(k)\|^2+\|x_i(k)-x^*\|^2  \\ 
    &-2 (x_i(k) - x^* )^{T}\\
    \leq{} & \|\epsilon_{i}\|^2+\alpha^2\|g(k)\|^2+\|x_i(k)-x^*\|^2\\
    &- \alpha c_1\|g(k)\|^2-\alpha c_2\|x_i(k)-x^*\|^2,
    \end{aligned}
\end{equation*}
where the values of $c_1$ and $c_2$ are respectively \cite{nesterov2013introductory}:
$$c_1= \frac{2}{\mu + L}\quad \text{and} \quad  c_2=\frac{2\mu L}{\mu +L}. $$ 
Hence, Equation \eqref{eq:individual-distance} can be upper bounded by:
\begin{equation*}
\begin{aligned}
\| x_i(k+1) - x^* - \epsilon_{i} \|^2 ={} & \| x_i(k) - x^* - \epsilon_{i} - \alpha g(k)\|^2, \\
\leq{} &\| x_i(k) - x^*\|^2 + \| \epsilon_{i}\|^2 + \alpha^2\| g(k)\|^2 \\ 
      & +\|\epsilon_{i}\|^2 + \alpha^2\| g(k)\|^2 {+} \| x_i(k) {-} x^*\|^2 \\
      & - \alpha c_1\|g(k)\|^2 -\alpha c_2\|x_i(k){-}x^*\|^2 \\
      & - \alpha c_1\|g(k)\|^2-\alpha c_2\|x_i(k)-x^*\|^2,\\
      ={} & (2-2\alpha c_2)\| x_i(k) - x^*\|^2 \\ 
&+ (2\alpha^2-2\alpha c_1)\|g(k)\|^2 + 2\|\epsilon_{i}\|^2.
\end{aligned}
\end{equation*}
Since $\|\epsilon_i\|^2$ is positive, the term $(2\alpha^2-2\alpha c_1)$ is negative when $\alpha <c_1$ and using the fact that $0<(2-2\alpha c_2)<1$, which we showed in Lemma \ref{lemmacentralized3}, we obtain
\begin{equation}\label{eqn:upperbound31}
\| x_i(k+1){-}x^*{-}\epsilon_i \|^2\leq (2{-}2\alpha c_2)\|x_i(k){-}x^*\|^2{+}2\|\epsilon_i\|^2.
\end{equation}
Since $x_i(k+1) - x^* < 0 <\epsilon_i$, we have that
\begin{equation}
 \| x_i(k+1) - x^* - \epsilon_i\|^2 > \| x_i(k+1) - x^* \|^2.\\
\label{eqn:upperbound32}
\end{equation}
From equations \eqref{eqn:upperbound31} and \eqref{eqn:upperbound32}, we obtain:
\begin{equation*}
\| x_i(k+1){-}x^*\|^2 \leq (2{-}2\alpha c_2)\| x_i(k){-}x^*\|^2{+}2\|\epsilon_i\|^2,
\end{equation*}
and by the recursive relationship, we obtain:
\begin{equation*}
\| x_i(k){-}x^*\|^2\leq (2{-}2\alpha c_2)^k\| x_i(0){-}x^*\|^2{+}2\|\epsilon_i\|^2,
\end{equation*}
from which we conclude that
\begin{equation*}
\| x_i(k){-}x^*\|\leq (2{-}2\alpha c_2)^{\frac{k}{2}}\| x_i(0){-}x^*\|{+}\sqrt{2}\|\epsilon_i\|.
\end{equation*}
Therefore, the individual iterates converge to the neighborhood of the optimal solution, $x^*$.
\end{proof}
Lemma \ref{lemmaindividualagent2} illustrates the deviations of individual agents from the optimal solution and the bound indicates the chosen attack vector affects the neighborhood of convergence. While in Lemma \ref{lemmacentralized3} allows adversarial agents to coordinate and use a uniform attack vector $\epsilon$, the result in Lemma does not require cooperation or the use of a uniform attack vector. 

\section{Numerical Experiments}\label{sec:Numerical}
In this section, we illustrate our theoretical results of Lemmas \ref{lemmacentralized3} and \ref{lemmaindividualagent2} over a network of $n= 10$ agents and $100$ iterations where the objective is to solve the unconstrained problem
\begin{equation}\label{eqn:simulationquadratic}
    f(x)=\frac{1}{2}x^Tx
\end{equation}
in a distributed way. Clearly the cost function is strongly convex with strong convexity parameter $\mu = 1$. Also, its gradient has Lipschitz continuity parameter $L = 1$.
By inspection, the optimal solution of Problem  \eqref{eqn:simulationquadratic} is $x^*=0$. 
We will show how the choices of attack vectors of different magnitudes and agents' initial estimates influence convergence to the neighborhood of the optimal solution. In the illustrations to follow, entries of the attack vector was drawn uniform distribution over the interval $(0,1)$. For the complete and general graph cases below, we use a step size of $\alpha = 0.6$.

\subsection{Complete Graph Case with Common Attack Vector}
We begin with the case when the communication network is a complete graph, the case in which the adversarial agents perturb their local iterates with a common attack vector. We define the error as the distance between the average iterate and the optimal solution and present the error convergence in Figure \ref{fig:iteration-8-non} to \ref{fig:iteration-1-non}. For the plot in Figure \ref{fig:iteration-1-non}, we assumed the number of non-adversarial nodes was $1$ with $9$ adversarial nodes.

In Figure \ref{fig:iteration-5-non}, we illustrate convergence of the error when there are $5$ adversarial and $5$ non-adversarial nodes in the $10$-node network. As can be observed, the gap between the upper bound of the error (that is the neighborhood of convergence), and the actual error obtained increased, indicating that with an increased number of non-adversarial nodes, a closer solution to the optimal solution is obtained. 
Figure \ref{fig:iteration-8-non} contains the plot for the scenario with $8$ non-adversarial nodes and $2$ adversarial nodes, which shows a further reduction in the actual error obtained. In the three figures, we can also observe that the average iterate of all agents in the network stays close to the optimal solution. And as the proportion of adversarial agents in the network increase, the average value moves away from zero.
\begin{figure}[h]
    \centering
    \includegraphics[width=8cm]{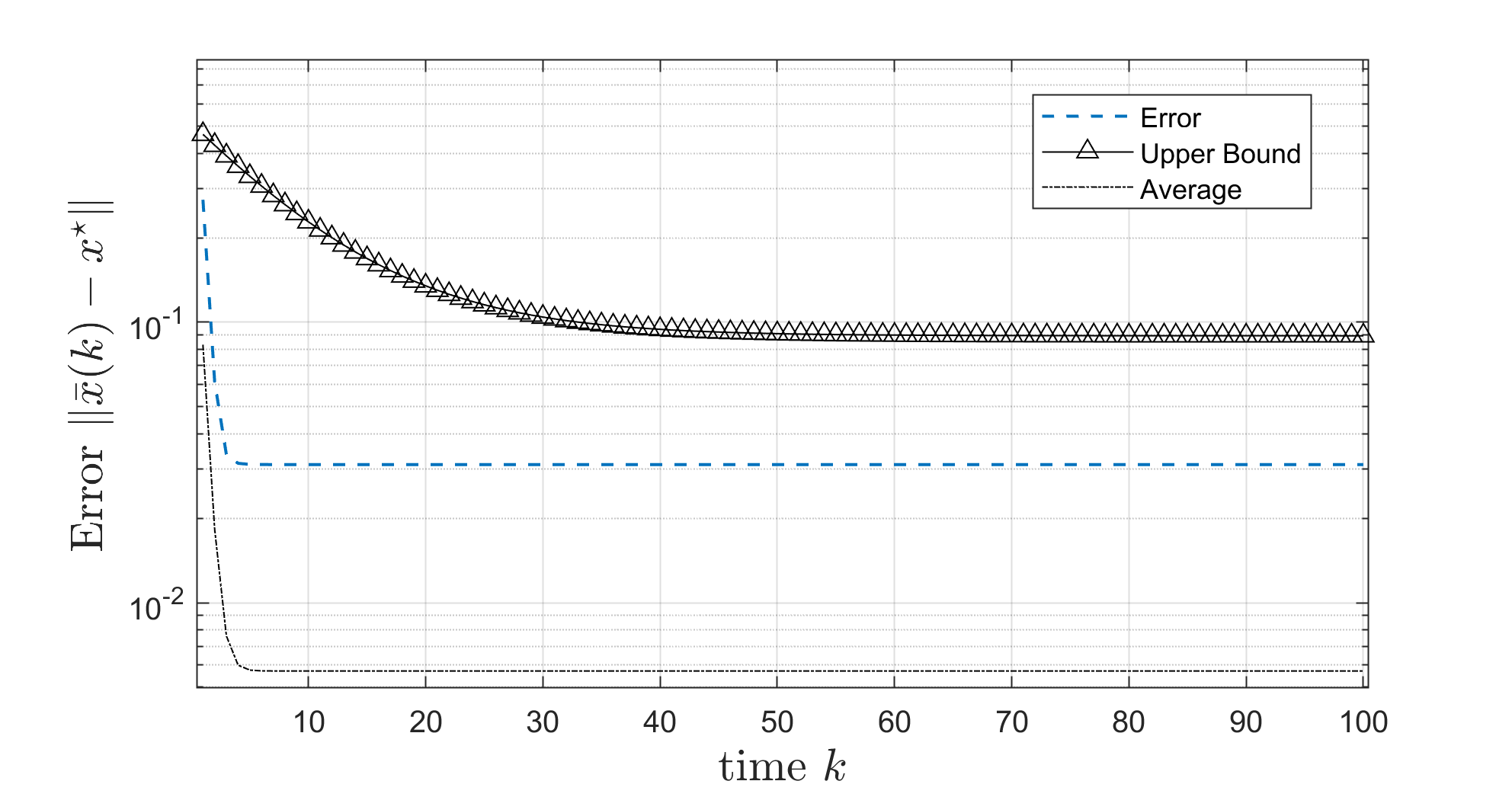}
    \caption{Simulations for $8$ non-adversarial agent and $2$ malicious agents.}
    \label{fig:iteration-8-non}
\end{figure}
\begin{figure}[h]
    \centering
    \includegraphics[width=8cm]{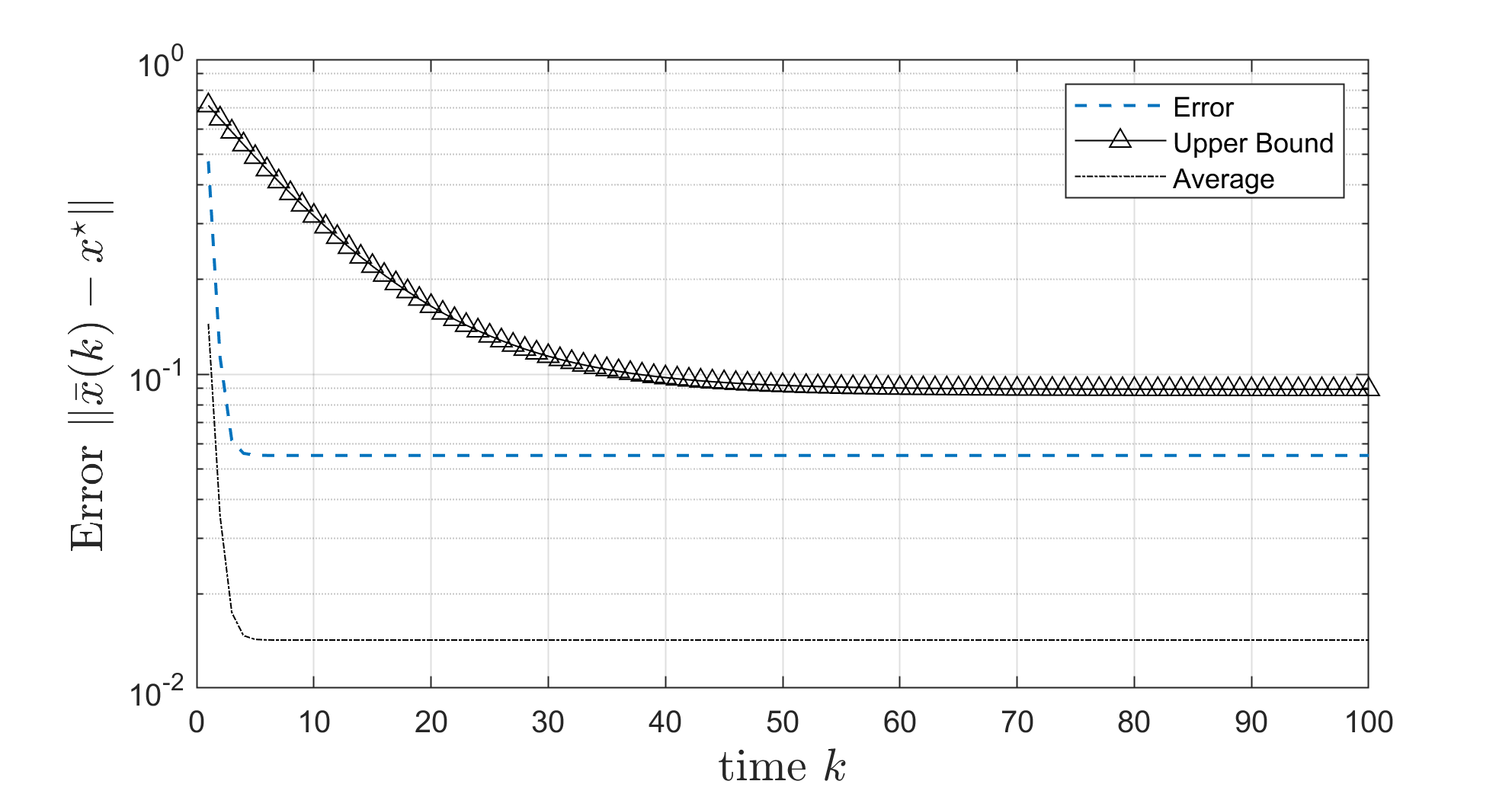}
    \caption{Simulations for $5$ non-adversarial agent and $5$ malicious agents.}
    \label{fig:iteration-5-non}
\end{figure}
\begin{figure}[h]
    \centering
    \includegraphics[width=8cm]{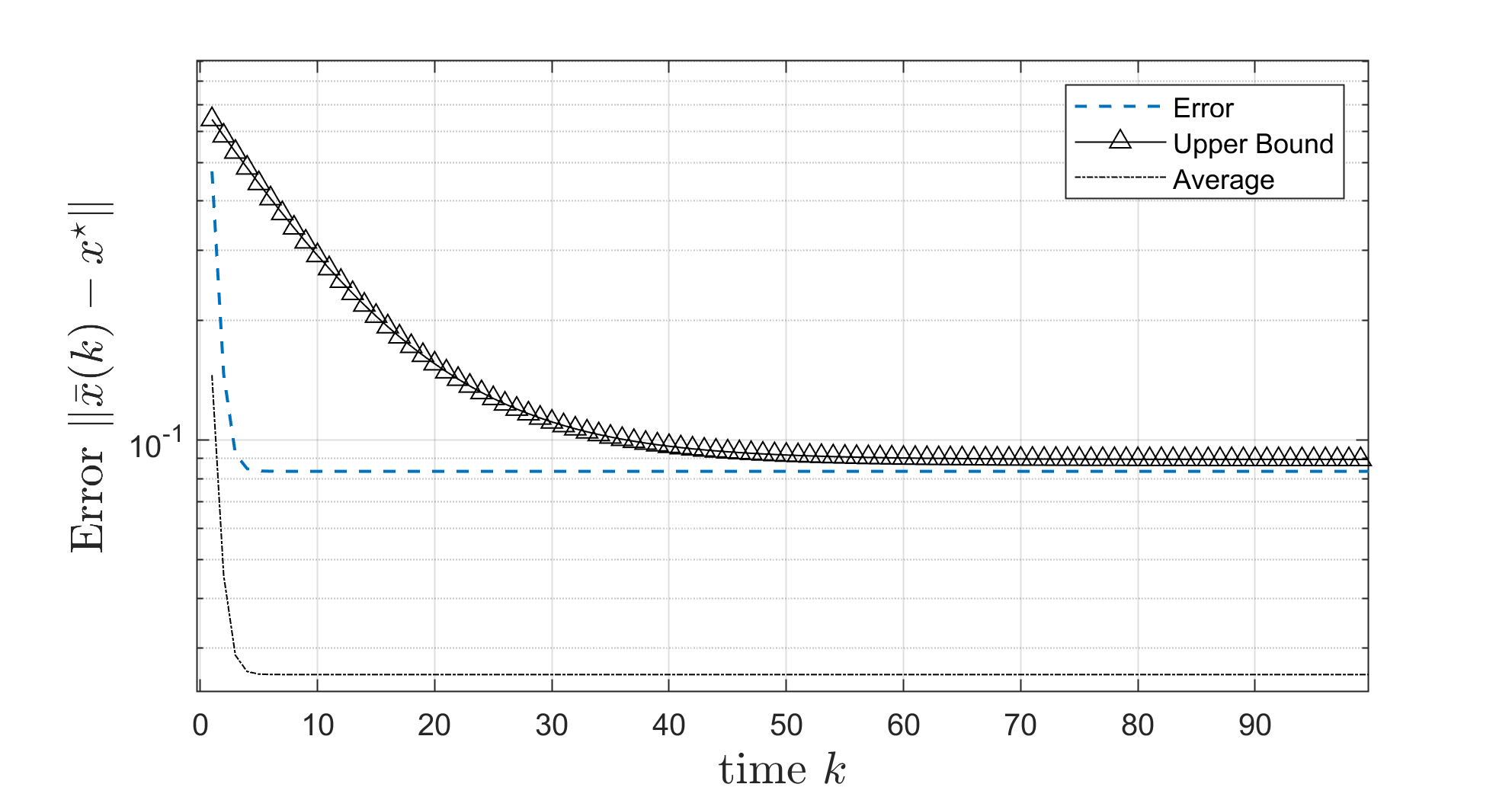}
    \caption{Simulations for $1$ non-adversarial agent and $9$ malicious agents.}
    \label{fig:iteration-1-non}
\end{figure}

\subsection{General Graph Case with Different Attack Vectors}
Similar illustrations are made in Figures \ref{fig:iteration-gen-8-non} to \ref{fig:iteration-gen-3-non}  where we vary the number of malicious nodes for the general (non-complete) graph case comprising $n=10$ agents solving Problem \eqref{eqn:simulationquadratic}. We also show the error evolution for different proportions of malicious to non-malicious nodes. In this case, each malicious node perturbs their local estimate with a different attack vector $\epsilon$ at each time step. Figure \ref{fig:iteration-gen-3-non} shows the case with $3$ non-adversarial and $7$ adversarial nodes. 
\begin{figure}[h]
    \centering
    \includegraphics[width=8cm]{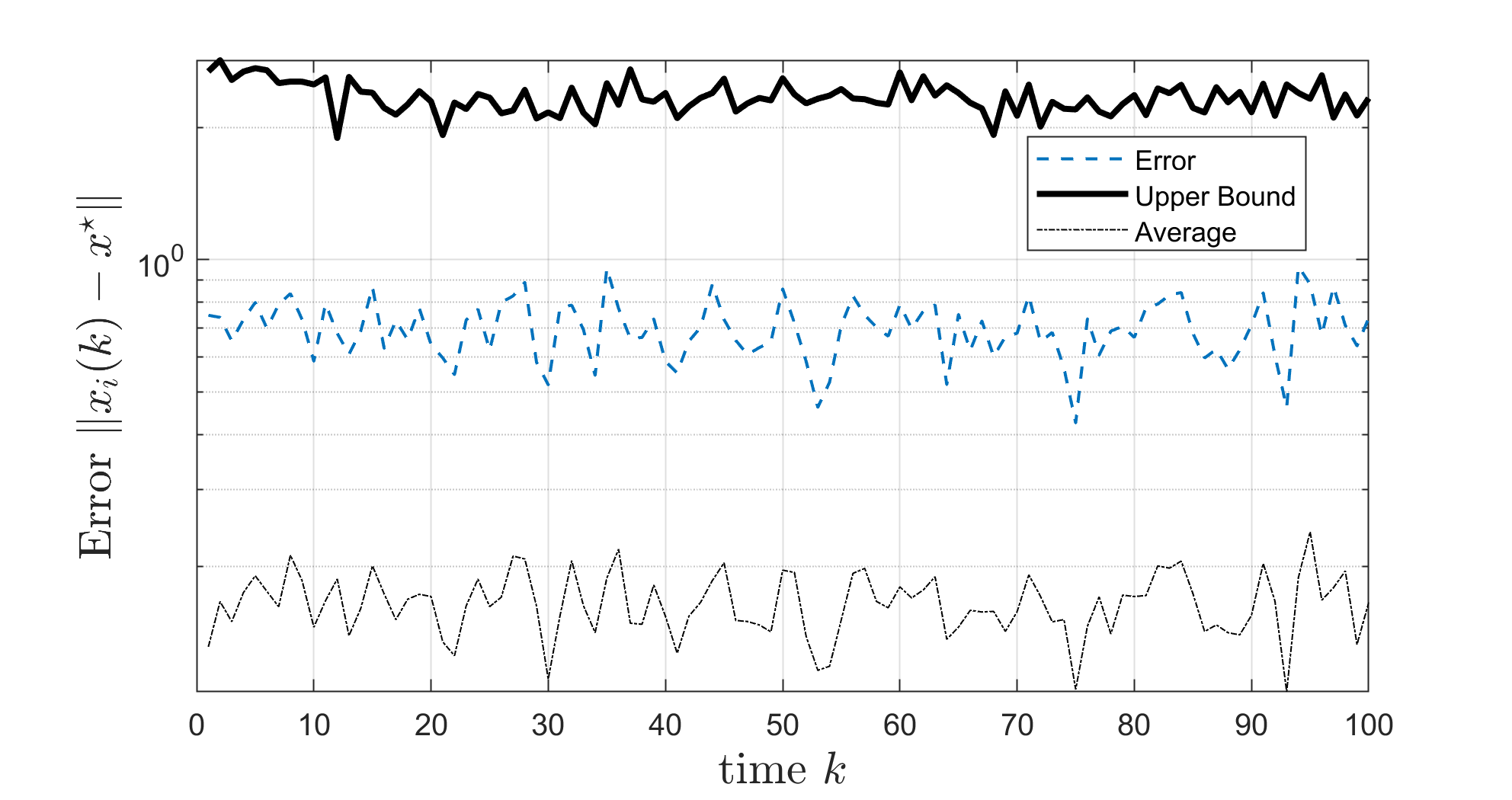}
    \caption{Simulations for $8$ non-adversarial agent and $2$ malicious agents.}
    \label{fig:iteration-gen-8-non}
\end{figure}
\begin{figure}[h]
    \centering
    \includegraphics[width=8cm]{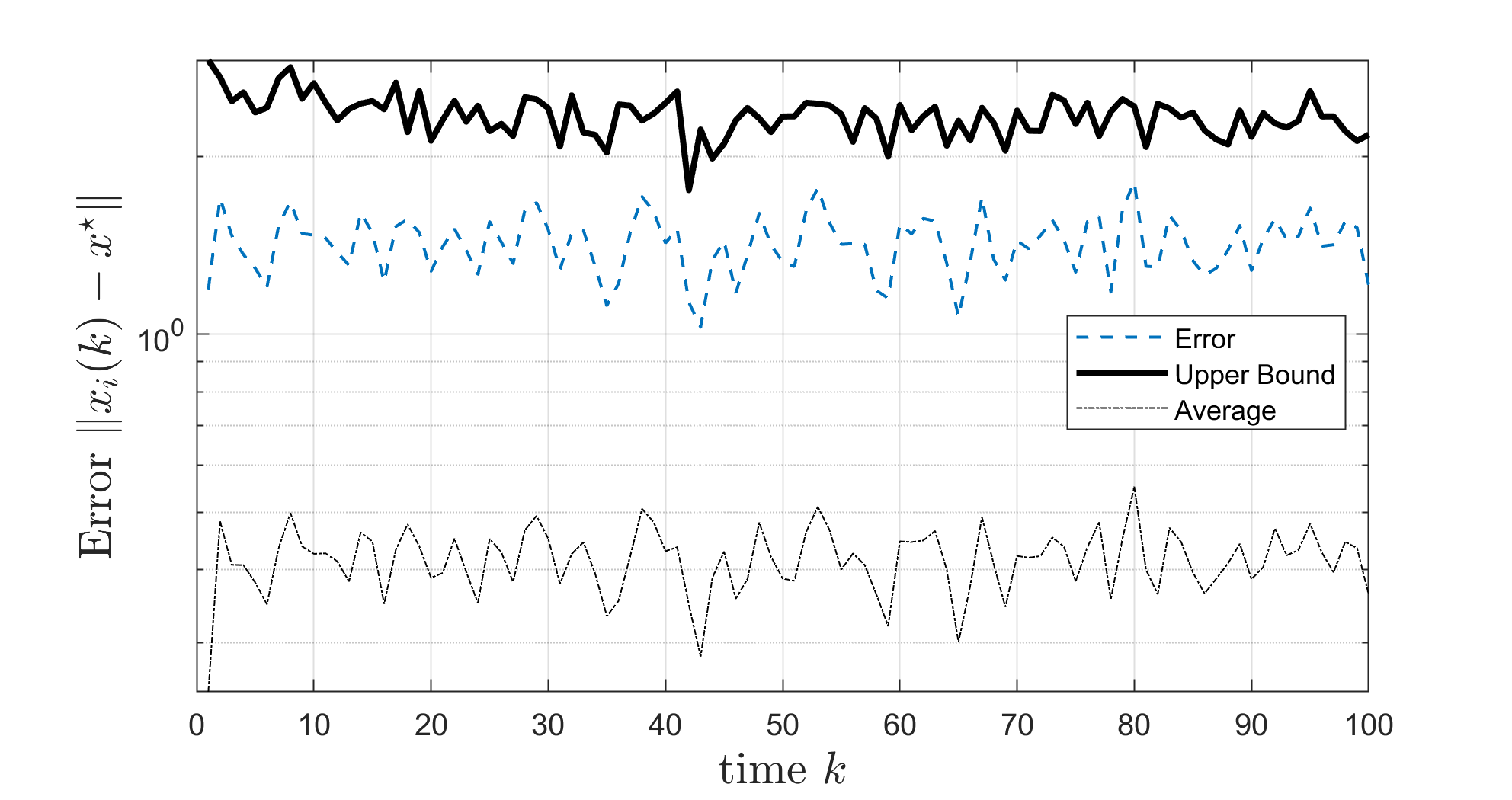}
    \caption{Simulations for $5$ non-adversarial agent and $5$ malicious agents.}
    \label{fig:iteration-gen-5-non}
\end{figure}
\begin{figure}[h]
    \centering
    \includegraphics[width=8cm]{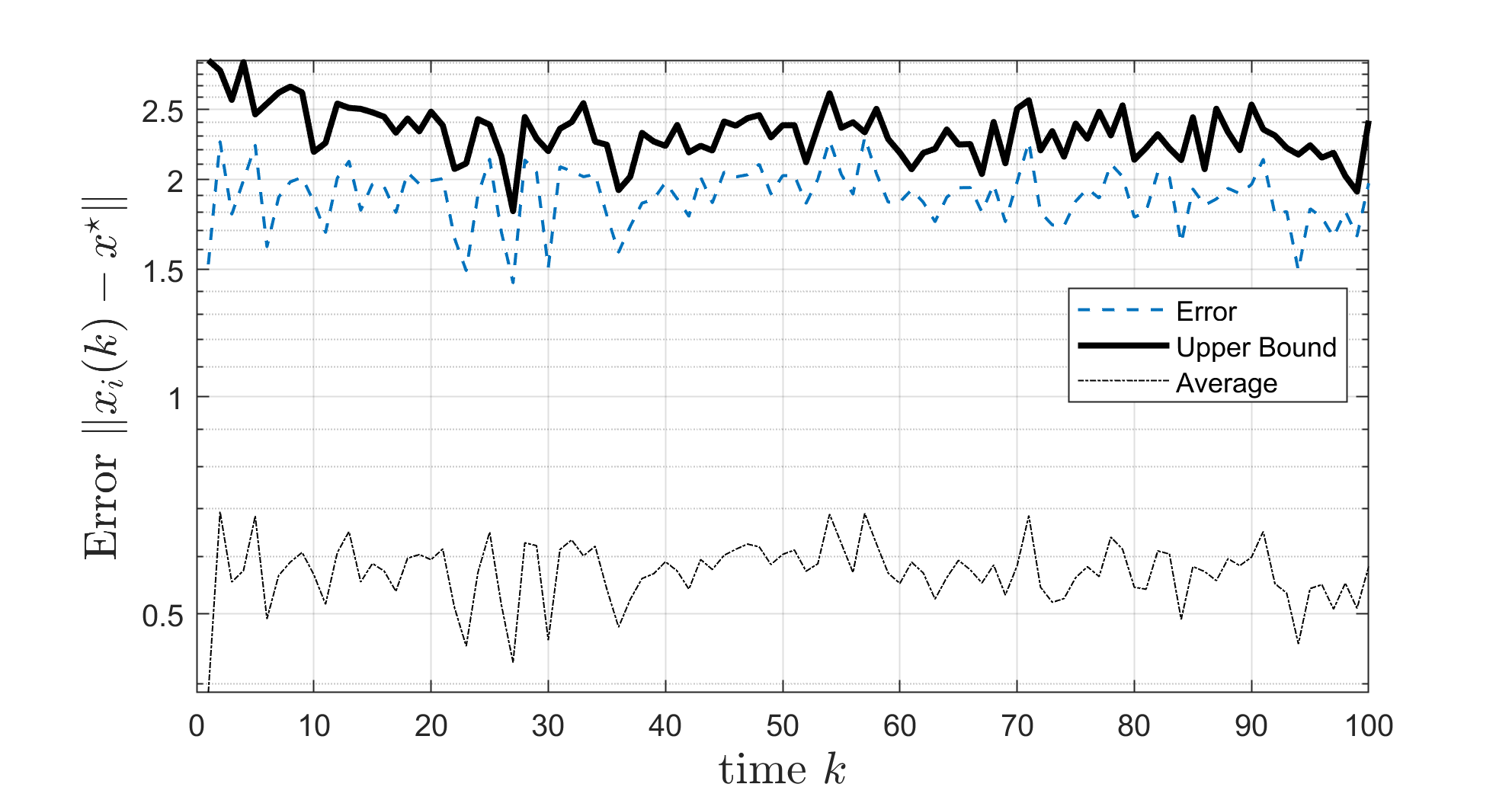}
    \caption{Simulations for $3$ non-adversarial agent and $7$ malicious agents.}
    \label{fig:iteration-gen-3-non}
\end{figure}
Figure \ref{fig:iteration-gen-5-non} shows the case with $5$ non-adversarial and $5$ adversarial nodes; and Figure \ref{fig:iteration-gen-8-non} shows the case comprising $8$ non-adversarial nodes and $2$ malicious nodes. From the figures, we can observe that as the ratio of malicious nodes in the network increases, the error increases towards the bound of the convergence neighborhood. In addition, the average iterates of agents in the network stays close to the optimal solution of zero. And as the proportion of adversarial nodes in the network increase, the average of the iterates move away from zero.  This outcome is intuitive and expected, since the presence of more agents causing disruption to the distributed consensus-based gradient algorithm would cause a greater deviation from the optimal solution. 


\section{Conclusions}
\label{sec:conclude}
This paper considered a distributed optimization problem over a network of agents in which some agents exude adversarial behavior by perturbing their local iterates of the decision variable before sharing it with neighboring agents at each time step. We established conditions needed for the iterates of the agents to converge to a neighborhood of the optimal solution;  and demonstrated our results via simulations. Convergence of the agents' iterate to a neighborhood of the optimal solution depended on not only an appropriate choice of the step size, but also on the distance between the initial iterate and the optimal solution being less than the attack vector. The simulations show that as the number of adversarial agents in the network increase, the convergence neighborhood increases towards the upper bound in the Lemmas 1 and 2.

\bibliographystyle{IEEEtran}
{\small
\bibliography{mybib1.bib}}

\begin{thebibliography}{10}
\providecommand{\url}[1]{#1}
\csname url@samestyle\endcsname
\providecommand{\newblock}{\relax}
\providecommand{\bibinfo}[2]{#2}
\providecommand{\BIBentrySTDinterwordspacing}{\spaceskip=0pt\relax}
\providecommand{\BIBentryALTinterwordstretchfactor}{4}
\providecommand{\BIBentryALTinterwordspacing}{\spaceskip=\fontdimen2\font plus
\BIBentryALTinterwordstretchfactor\fontdimen3\font minus
  \fontdimen4\font\relax}
\providecommand{\BIBforeignlanguage}[2]{{%
\expandafter\ifx\csname l@#1\endcsname\relax
\typeout{** WARNING: IEEEtran.bst: No hyphenation pattern has been}%
\typeout{** loaded for the language `#1'. Using the pattern for}%
\typeout{** the default language instead.}%
\else
\language=\csname l@#1\endcsname
\fi
#2}}
\providecommand{\BIBdecl}{\relax}
\BIBdecl

\bibitem{boyd2011distributed}
S.~Boyd, N.~Parikh, and E.~Chu, \emph{Distributed optimization and statistical
  learning via the alternating direction method of multipliers}.\hskip 1em plus
  0.5em minus 0.4em\relax Now Publishers Inc, 2011.

\bibitem{nesterov2013introductory}
Y.~Nesterov, \emph{Introductory lectures on convex optimization: A basic
  course}.\hskip 1em plus 0.5em minus 0.4em\relax Springer Science \& Business
  Media, 2013, vol.~87.

\bibitem{maybury2006detecting}
M.~Maybury, ``Detecting malicious insiders in military networks,'' MITRE CORP
  BEDFORD MA, Tech. Rep., 2006.

\bibitem{ravi2019case}
N.~Ravi, A.~Scaglione, and A.~Nedi{\'c}, ``A case of distributed optimization
  in adversarial environment,'' in \emph{ICASSP 2019-2019 IEEE International
  Conference on Acoustics, Speech and Signal Processing (ICASSP)}.\hskip 1em
  plus 0.5em minus 0.4em\relax IEEE, 2019, pp. 5252--5256.

\bibitem{sundaram2010distributed}
S.~Sundaram and C.~N. Hadjicostis, ``Distributed function calculation via
  linear iterative strategies in the presence of malicious agents,'' \emph{IEEE
  Transactions on Automatic Control}, vol.~56, no.~7, pp. 1495--1508, 2010.

\bibitem{baras2019trust}
J.~S. Baras and X.~Liu, ``Trust is the cure to distributed consensus with
  adversaries,'' in \emph{2019 27th Mediterranean Conference on Control and
  Automation (MED)}.\hskip 1em plus 0.5em minus 0.4em\relax IEEE, 2019, pp.
  195--202.

\bibitem{zhao2017resilient}
C.~Zhao, J.~He, and Q.-G. Wang, ``Resilient distributed optimization algorithm
  against adversary attacks,'' in \emph{2017 13th IEEE International Conference
  on Control \& Automation (ICCA)}.\hskip 1em plus 0.5em minus 0.4em\relax
  IEEE, 2017, pp. 473--478.

\bibitem{sundaram2016secure}
S.~Sundaram and B.~Gharesifard, ``Secure local filtering algorithms for
  distributed optimization,'' in \emph{2016 IEEE 55th Conference on Decision
  and Control (CDC)}.\hskip 1em plus 0.5em minus 0.4em\relax IEEE, 2016, pp.
  1871--1876.

\bibitem{lu2020distributed}
A.-Y. Lu and G.-H. Yang, ``Distributed secure state estimation in the presence
  of malicious agents,'' \emph{IEEE Transactions on Automatic Control}, 2020.

\bibitem{marano2008distributed}
S.~Marano, V.~Matta, and L.~Tong, ``Distributed detection in the presence of
  byzantine attacks,'' \emph{IEEE Transactions on Signal Processing}, vol.~57,
  no.~1, pp. 16--29, 2008.

\bibitem{yan2012vulnerability}
Q.~Yan, M.~Li, T.~Jiang, W.~Lou, and Y.~T. Hou, ``Vulnerability and protection
  for distributed consensus-based spectrum sensing in cognitive radio
  networks,'' in \emph{2012 Proceedings IEEE INFOCOM}.\hskip 1em plus 0.5em
  minus 0.4em\relax IEEE, 2012, pp. 900--908.

\bibitem{zhang2014distributed}
J.~Zhang, P.~Jaipuria, A.~Chakrabortty, and A.~Hussain, ``A distributed
  optimization algorithm for attack-resilient wide-area monitoring of power
  systems: Theoretical and experimental methods,'' in \emph{International
  Conference on Decision and Game Theory for Security}.\hskip 1em plus 0.5em
  minus 0.4em\relax Springer, 2014, pp. 350--359.

\bibitem{kailkhura2015consensus}
B.~Kailkhura, S.~Brahma, and P.~K. Varshney, ``Consensus based detection in the
  presence of data falsification attacks,'' \emph{arXiv preprint
  arXiv:1504.03413}, 2015.

\bibitem{berahas2018balancing}
A.~S. Berahas, R.~Bollapragada, N.~S. Keskar, and E.~Wei, ``Balancing
  communication and computation in distributed optimization,'' \emph{IEEE
  Transactions on Automatic Control}, vol.~64, no.~8, pp. 3141--3155, 2018.

\end{thebibliography}

\end{document}